\documentclass[11pt]{article}

\usepackage{graphicx}
\usepackage{amssymb}
\usepackage{amsmath}
\usepackage{amsthm}
\usepackage{latexsym}
\usepackage[english]{babel}
\usepackage{appendix}
\usepackage{tikz}
\usetikzlibrary{matrix,arrows,decorations.pathmorphing}

\usepackage{tikz-cd}
\usepackage{hyperref}

\usepackage[nottoc,notlot,notlof]{tocbibind}

\newtheorem{theorem}{Theorem}[section]

\newtheorem{lemma}{Lemma}[section]
\newtheorem{proposition}{Proposition}[section]

\newtheorem{corollary}{Corollary}[section]

\def \C {\mathbb{C}}
\def \R {\mathbb{R}}
\def \Z {\mathbb{Z}}

\def \N {\mathbb{N}}

\begin{document}

\title{A conformally invariant Dirac-type equation on compact spin manifolds: the effect of the geometry.}

\author{ Ali Maalaoui$^{(1)(2)}$ \& Vittorio Martino$^{(3)}$}
\begin{NoHyper}
\addtocounter{footnote}{1}
\footnotetext{Department of Mathematics, Clark University, 950 Main Street, Worcester, MA 01610, USA. E-mail address: \tt{amaalaoui@clarku.edu}}
\addtocounter{footnote}{1}
\footnotetext{Department of Mathematics, MIT, 77 Massachusetts Avenue
Cambridge, MA 02139-4307. E-mail address: \tt{maala650@mit.edu}}
\addtocounter{footnote}{1}
\footnotetext{Dipartimento di Matematica, Alma Mater Studiorum - Universit\`a di Bologna. E-mail address: \tt{vittorio.martino3@unibo.it}}
\end{NoHyper}

\date{}
\maketitle

\vspace{5mm}

{\noindent\bf Abstract}  {\small Let $(M,g)$ be a closed Riemannian spin manifold of dimension greater or equal than four. Here we consider a generalized conformally invariant equation involving the Dirac operator, where the nonlinear term is given by the convolution with a suitable power of the Green’s function of the conformal Laplacian. In this situation one can show that there exists a conformally invariant Yamabe type constant, which corresponds to the energy of the non-trivial ground state solutions; moreover this constant is always less or equal than the one of the round sphere: in case the inequality is strict, then the considered equation admits a non-trivial solution. The main result of this paper is that the inequality is always strict, unless $(M,g)$ is conformal to the round sphere; therefore a non-trivial ground state solution always exists. In particular, since in dimension four this equation coincide with the conformal Dirac-Einstein system, our result a ground state solution for the system. We point out that aside from some perturbative or special cases, this presents the first general existence result for the conformal Dirac-Einstein equations in dimension four.}

\vspace{5mm}

\noindent
{\small Keywords: Dirac operator, Convolution non-linearity, Conformal Dirac-Einstein.

\vspace{4mm}


\vspace{4mm}


\section{Introduction and motivation}
Let us consider a Riemannian manifold $(M,g)$, closed (compact, without boundary), of dimension $n \geq 3$; let also $\Sigma_{g} M$ denote the related spinor bundle (see Section 2). In \cite{MMM} the authors studied a conformally invariant non-local problem which, in particular, coincides with the conformal Dirac-Einstein system in dimension 4. Indeed, the conformal Dirac-Einstein system takes the form  
\begin{equation}\label{DE4}
\left\{\begin{array}{ll}
L_{g} u=\displaystyle\frac{\lambda}{n-2}|\psi|^{2} u^{\frac{4-n}{n-2}}\\
\\
D_{g}\psi=\lambda u^{\frac{2}{n-2}} \psi
\end{array}.
\right.
\end{equation}
where $L_{g}$ is the conformal Laplacian, $D_{g}$ is the Dirac operator acting on spinors $\psi\in \Sigma_{g} M$, with $\lambda$ a real parameter. Starting from the pioneering work \cite{BrillWheeler1957} and then \cite{Fin}, this conformal system has been quite studied in recent years: see for instance (\cite{BM, MaalaouiMartino2019, MMX}) in dimension 3 (\cite{BorrelliMaalaouiMartino2023} the case with boundary), \cite{YX} for all dimensions via a flow approach; in fact, it was surprising to find that there have been even earlier investigations of the conformal Dirac-Einstein system in \cite{Br} by T. Branson.

\noindent
In particular when $n=4$ and $G_{g}$ is the Green's function of $L_{g}$, the system can be written as a single equation (we assume $\lambda=1$)
\begin{equation}\label{eqo}
D_{g}\psi=\Big(\int_{M}G_{g}(\cdot,y)|\psi(y)|^{2}\ dv_{g}(y)\Big) \psi = (G_g * |\psi|^{2})\psi.
\end{equation}
Equations of this kind (from an analysis point of view, focusing mainly on the singular behaviour of the Green's function rather than conformal geometry) were treated in several works in the literature because of their physical significance: for instance, we refer the reader to \cite{BGDB,DWX,ES,GG,Enno,Lieb1,Lieb2,MS} and the references therein.\\
At this point, one way of extending $(\ref{eqo})$ to higher dimensions, preserving the conformal invariance (and hence the geometric nature) of the equation, was to study the following problem:

\begin{equation}\label{eq}
D_g\psi=(G_g^{s}*|\psi|^{2})\psi,
\end{equation}
where we denoted by $$(G_g^{s}*f)(x) =\int_{M}G_g^{s}(x,y)f(y) \; dv_{g}(y) \; $$
the convolution of a given function $f$ with the Green's function $G_g^{s}$ of the conformal fractional Laplacian  $P_{g}^{s}$, of order $2s=n-2$, see for example \cite{CC, Mar, GJMS, GZ}.\\
Since the equation has a variational structure, the existence of solutions was investigated from that perspective, by studying the energy functional
\begin{align}\label{energyfunctional}
J_{g}(\psi) & = \frac{1}{2}\int_{M}\langle D_g\psi,\psi\rangle \ dv_{g}-\frac{1}{4}\int_{M} (G_g^{s}*|\psi|^{2})|\psi|^{2}\; dv_{g} \notag \\
& = \frac{1}{2}\int_{M}\langle D_g\psi,\psi\rangle \ dv_{g}-\frac{1}{4}\int_{M\times M} G_g^{s}(x,y)|\psi(y)|^{2} \; |\psi(x)|^{2}\; dv_{g}(y)\; dv_{g}(x) \;,
\end{align}
where $\langle \cdot ,\cdot\rangle$ is the canonical Hermitian metric defined on  $\Sigma_g M$.\\
The main focus in \cite{MMM} was on the study of the non-compactness of this variational problem and the bubbling phenomena that manifests through this lack of compactness, which led to the following theorem:

\begin{theorem}[\cite{MMM}]\label{thm1}
Let $(M,g)$ be a closed Riemannian spin manifold of dimension $n \geq 3$. Let us assume that $(M,[g])$ has a positive s-Yamabe constant $Y_s(M,[g])$ (see below) and let $(\psi_{k})_{k\in \N}$ be a Palais-Smale sequence for $J_{g}$ at level $c\geq 0$. Then there exist $\psi_{\infty}\in C^{\infty}(M,\Sigma_{g} M)$, a solution of $(\ref{eq})$, $m$ sequences of points $x_{k}^{1},\cdots, x_{k}^{m} \in M$ such that $\lim_{k\to \infty}x_{k}^{j}= x^{j}\in M$, for $j=1,\dots,m$ and $m$ sequences of real numbers $R_{k}^{1},\cdots, R_{k}^{m}$ converging to zero, such that:
\begin{itemize}
\item[i)]  $\displaystyle \psi_{k}=\psi_{\infty}+\sum_{j=1}^{m}\phi_{k}^{j}+o(1)$ in $H^{\frac{1}{2}}(\Sigma M)$,
\item[ii)] $\displaystyle J_{g}(\psi_{k})=J_{g}(\psi_{\infty})+ 
    \sum_{j=1}^{m}J_{g_{\mathbb{R}^{n}}}(\Psi_\infty^{j})+o(1)$,
\end{itemize}
where
$$\phi_{k}^{j}=(R_{k}^{j})^{-1}\beta_{j}\sigma_{k,j}^{*}(\Psi_\infty^{j}) ,$$
with $\sigma_{k,j}=(\rho_{k,j})^{-1}$ and $\rho_{k,j}(\cdot)=exp_{x_{k}^j}(R_{k}^j \cdot)$ is the exponential map defined in a suitable neighborhood of $\R^{n}$.
Also, here $\beta_{j}$ is a smooth compactly supported function, such that $\beta_{j}=1$ on $B_{1}(x^{j})$ and $supp(\beta_{j})\subset B_{2}(x^{j})$ and $\Psi_\infty^{j}$ is the solution to our equations (\ref{eq}) on $\mathbb{R}^n$ with its Euclidian metric $g_{\R^n}$.
\end{theorem}

\noindent
Here we summarize the main ingredients that were needed in the proof of the previous result:
\begin{itemize}
\item[i)] $G_{g}^{s}$ is symmetric and there exists $C > 0$ such that in some small neighborhood in normal coordinates 
$$G_{g}^{s}(x,y)=\frac{h_{n}}{|x-y|^{2}}+r(x,y) \;, $$
where
$$h_{n}=\frac{1}{2^{n-2}\pi^{\frac{n}{2}}\Gamma(\frac{n}{2}-1)}\; , \qquad |r(x,y)|\leq \frac{C}{|x-y|} \; .$$
\item[ii)] $G_{g}^{s}$ is conformally invariant in the sense that the functional $J_{g}$ is conformally invariant.
\item[iii)] On $\R^{n}$, $G_{\R^{n}}^{s}(x,y)=\frac{h_{n}}{|x-y|^{2}}.$
\end{itemize}
Therefore, it is equally reasonable to think about a different extension of $(\ref{eqo})$ consisting of replacing $G^{s}_{g}$ with any potential $V_g$ satisfying the points i)-ii)-iii) above as long as, in dimension 4, $V_g=G_{g}$. Hence one can consider the potential $V_{g}$ defined by 
$$V_{g}=d_{n} \, (G_{{g}})^{\frac{2}{n-2}}, \quad d_{n}=\frac{h_{n}}{b_{n}^{\frac{2}{n-2}}}, \quad b_{n}=\frac{1}{n(n-2)\omega_{n}}=\frac{\Gamma(\frac{n}{2}+1)}{n(n-2)\pi^{\frac{n}{2}}} \; .$$
This choice obviously satisfies the points i) and iii). Moreover, it is compatible with the models in \cite{BGDB,DWX,ES,GG,Enno,Lieb1,Lieb2,MS}. The last point that we need to check is the conformal invariance in ii), with respect to the related functional $J_{g}$, with the potential $V_g$:
\begin{equation}\label{modifiedJg}
  J_{g}(\psi)=\frac{1}{2}\int_{M}\langle D \psi,\psi\rangle   \ dv_{g} -\frac{1}{4}\int_{M\times M} V_{g}(x,y)|\psi(x)|^{2} |\psi(y)|^{2} \ d_{g}v(x) dv_{g}(y) \; .
\end{equation}
So let us consider a conformal change of the metric
\begin{equation}\label{changemetric}
  \widetilde{g} = u^{\frac{4}{n-2}}g,  \; 0<u\in C^{\infty}(M) \; .
\end{equation}
Given a spinor $\psi \in \Sigma_{g} M$, we set 
$$\widetilde{\psi} = u^{\frac{1-n}{n-2}} \psi \in \Sigma_{\widetilde{g}} M \; ,$$
where we implicitly understand the action of a canonical isometric isomorphism between the spinor bundles $\Sigma_{\widetilde{g}} M$ and $\Sigma_{g} M$ (see Section 2 below). In this way, we have the conformal change of the Dirac operator
$$
D_{ \widetilde{g}}\widetilde{\psi}  = u^{-\frac{n+1}{n-2}} D_{g} \psi \; .
$$
On the other hand, the potential $V_{g}$ satisfies the following transformation under conformal change:
$$
V_{\widetilde{g}}(x,y)=u(x)^{-\frac{2}{n-2}} u(y)^{-\frac{2}{n-2}}V_{g}(x,y) \; .
$$
Now, taking into account the change of the volume form
$$dv_{\widetilde{g}}  = u^{\frac{2n}{n-2}}  dv_{g} \; ,$$
we substitute in the modified functional $J_{g}$ and find
\begin{align*}
J_{\widetilde{g}}(\widetilde{\psi}) & := \frac{1}{2}\int_{M}\langle D_{\widetilde{g}}\widetilde{\psi},\widetilde{\psi}\rangle \ dv_{\widetilde{g}}-\frac{1}{4}\int_{M\times M} V_{\widetilde{g}}(x,y)|\widetilde{\psi}(y)|^{2} \; |\widetilde{\psi}(x)|^{2}\; dv_{\widetilde{g}}(y)\; dv_{\widetilde{g}}(x) \;,\\
& = \frac{1}{2}\int_{M}\langle D_{g}\psi,\psi\rangle \ dv_{g}-\frac{1}{4}\int_{M\times M} V_g(x,y)|\psi(y)|^{2} \; |\psi(x)|^{2} \;  u^{\frac{2+2n-2-2n}{n-2}} \; dv_{g}(y)\; dv_{g}(x) \; .
\end{align*}
Therefore $J_{\widetilde{g}}(\widetilde{\psi})=J_{g}(\psi)$ and the problem is conformally invariant.

\noindent
In this paper we will focus on proving the existence of non trivial critical points of the energy functional $J_{g}$ in (\ref{modifiedJg}); Namely, we want to exhibit solutions to the problem:
\begin{equation}\label{eqv}
D_{g}\psi=\left( V_{g}*|\psi|^{2}\right) \psi.
\end{equation}
In \cite{MMM}, the authors provided an Aubin-type inequality that would ensure the existence of least energy solutions to our problem. This can be summarized in the following corollary:
\begin{corollary}[\cite{MMM}]\label{cor1}
Under the assumptions of Theorem \ref{thm1}, there exists a conformally invariant constant $\overline{Y}(M,[g])>0$ with the following properties:
\begin{itemize}
\item[i)] $\overline{Y}(M,[g])$ corresponds to the energy of the non-trivial ground state solution of the equation (\ref{eqv}).
\item[ii)] $\overline{Y}(M,[g])\leq \overline{Y}(S^{n},[g_{0}])=\frac{\lambda^{+}(S^{n}, [g_{0}])^{2} Y_{\frac{n-2}{2}}(S^{n}, [g_{0}])}{4}$.
\item[iii)] If $\overline{Y}(M,[g])< \overline{Y}(S^{n},[g_{0}])$ then the problem $(\ref{eqv})$ has a non-trivial solution.
\end{itemize}
\end{corollary}

\noindent
Here $(S^{n}, [g_{0}])$ denotes the sphere with its standard round metric. Also, the constant $\lambda^{+}(M,[g])$ is the B\"{a}r-Hijazi-Lott invariant (see for instance \cite{Am1,Am2}), defined by
$$\lambda^{+}(M,[g])=\inf\left\{\lambda_{1}(D_{h})Vol(h)^{\frac{1}{n}}; h\in [g]\right\},$$
and $Y_{s}(M,[g])$ is the $s$-Yamabe constant corresponding to the continuous embedding of $L^{\frac{2n}{n-2s}}(M)$ in $H^{s}(M)$ defined by 
$$Y_{s}(M,[g])=\inf\left\{\frac{\int_{M}uP_{g}^{s}u\ dv_{g}}{\Big(\int_{M}u^{\frac{2n}{n-2s}}\ dv_{g}\Big)^{\frac{n-2s}{n}}}; u>0 \text{ and } u\in H^{s}(M)\right\},$$
where $P_{g}^{s}$ is the GJMS operator of order $2s$ (\cite{GJMS,GZ}).

\noindent
Our main result in the present paper can then be expressed as follows:
\begin{theorem}\label{main}
Let $(M,g)$ be a closed Riemannian spin manifold of dimension $n \geq 4$, with positive Yamabe invariant $Y(M,[g])>0$. Then $$\overline{Y}(M,[g])<\overline{Y}(S^{n},[g_{0}]),$$ 
unless $(M,[g])$ is conformal to the round sphere $(S^{n},[g_{0}])$ and in that case $$\overline{Y}(M,[g])=\overline{Y}(S^{n},[g_{0}]).$$ In particular, The equation $(\ref{eqv})$ always has a non-trivial solution.
\end{theorem}

\noindent
This result is similar to the one for the Yamabe problem, following \cite{LP}. Indeed, in the case of the Yamabe problem, the strategy of the proof consists in expanding the energy functional around an adequate test function. By doing so, two situations occur.
\begin{itemize}
\item The manifold is not locally conformally flat and the dimension $n\geq 6$: in this case the dominant term in the energy expansion is related to the Weyl tensor, bringing the energy below the critical energy level \cite{Au}.
\item The manifold is locally conformally flat or the dimension is below 6: this case relies on the Positive Mass Theorem \cite{S,W}; a more involved test function need to be manufactured and in the expansion of the energy, the mass term will be the dominant one, bringing the energy below the critical level.
\end{itemize}
We also refer to the work \cite{NSS} for a more comprehensive presentation on the influence of the different terms that appear in the asymptotic expansion of the energy functional.

\noindent
As a corollary of Theorem \ref{main}, we have:
\begin{corollary}
The conformal Dirac-Einstein system has a solution in dimension $n=4$ for manifolds with positive Yamabe invariant.
\end{corollary}

\noindent
To the best of our knowledge, aside from the perturbation result in \cite{GuidiMaalaouiMartino2021} (in dimension 3), this Corollary provides the only definite existence result for a solution to the conformal Dirac-Einstein system.

\noindent
Since the functional $J_{g}$ is strongly indefinite, classical minimization techniques do not apply and hence, a direct expansion of a test function would not lead to the existence of a critical point or to the estimate of a critical level. This constitutes a major difference from the classical scalar case. In fact, we will be following closely the functional framework introduced in \cite{YX1}. \\
The paper is organized as follows: in Section 2, we present the necessary preliminaries and tools that we will require to deal with the problem. In Section 3 we recall the functional setting that would allow us to estimate the energy level of the ground state solution. In Section 4, we focus on the construction of the test spinor including its gradient and energy estimates. We will distinguish two cases here, similar to those mentioned above, in particular: \\
- the case $n=4,5$ or $n\geq 6$ with $M$ not locally conformally flat;\\
- the case $n\geq 6$ with $M$ locally conformally flat.

\vspace{5mm}

{\noindent\bf Acknowledgment}

\noindent
The first author is supported by the AMS-Simons Research Enhancement Grant for PUI faculty under the project "Conformally Invariant
Non-Local Equations on Spin Manifolds".

\section{Preliminaries}

\noindent
In this section we recall some definitions and basic formulas that we will need in the sequel (see for instance \cite{KimFriedrich2000, MaalaouiMartino2022, MMM}). A spin structure on a riemannian manifold $(M,g)$ is a pair $(P_{Spin}(M,g),\sigma)$, where $P_{Spin}(M,g)$ is a $Spin(n)$-principal bundle and $\sigma : P_{Spin}(M,g)\to P_{SO}(M,g)$ is a 2-fold covering map, which restricts to a non-trivial covering $\kappa: Spin(n)\to SO(n)$ on each fiber. That is, the quotient of each fiber by $\Z_{2}$ is isomorphic to the frame bundle of $M$ and hence, the following diagram commutes:
\begin{center}

\tikzset{every picture/.style={line width=0.75pt}} 

\begin{tikzpicture}[x=0.75pt,y=0.75pt,yscale=-1,xscale=1]

\draw    (173,59) -- (380,59) ;
\draw [shift={(382,59)}, rotate = 180] [color={rgb, 255:red, 0; green, 0; blue, 0 }  ][line width=0.75]    (10.93,-3.29) .. controls (6.95,-1.4) and (3.31,-0.3) .. (0,0) .. controls (3.31,0.3) and (6.95,1.4) .. (10.93,3.29)   ;
\draw    (150,80) -- (247.36,147.86) ;
\draw [shift={(249,149)}, rotate = 214.88] [color={rgb, 255:red, 0; green, 0; blue, 0 }  ][line width=0.75]    (10.93,-3.29) .. controls (6.95,-1.4) and (3.31,-0.3) .. (0,0) .. controls (3.31,0.3) and (6.95,1.4) .. (10.93,3.29)   ;
\draw    (415,79) -- (314.63,149.85) ;
\draw [shift={(313,151)}, rotate = 324.78] [color={rgb, 255:red, 0; green, 0; blue, 0 }  ][line width=0.75]    (10.93,-3.29) .. controls (6.95,-1.4) and (3.31,-0.3) .. (0,0) .. controls (3.31,0.3) and (6.95,1.4) .. (10.93,3.29)   ;

\draw (86,43) node [anchor=north west][inner sep=0.75pt]    {$P_{Spin}( M,g)$};
\draw (388,43) node [anchor=north west][inner sep=0.75pt]    {$P_{SO}( M,g)$};
\draw (257,151) node [anchor=north west][inner sep=0.75pt]    {$( M,g)$};
\draw (275,30) node [anchor=north west][inner sep=0.75pt]    {$\sigma $};

\end{tikzpicture}
\end{center}

\noindent
We denote by $\mathbb{S}_{n}$ the unique (up to isomorphism) irreducible complex $Cl_{n}$-module such that $Cl_{n}\otimes \C \equiv End_{\C}(\mathbb{S}_{n})$ as a $\C$-algebra, where $Cl_{n}$ denotes the Clifford algebra of $\R^{n}$. This allows us to define the spinor bundle $\Sigma_{g} M$ as
$$\Sigma_{g} M:=P_{Spin}(M,g)\times_{\sigma} \mathbb{S}_{n}.$$
In fact, $\Sigma_{g} M$ is a Hermitian bundle equipped with a metric connection induced by the Levi-Civita connection on $TM$, that we will denote by $\nabla$. Moreover, there is a natural Clifford multiplication defined by the action of $TM$ on $\Sigma_{g} M$. We can summarize the main properties of the spinor bundle in the following few points (here ``$\cdot$'' denotes the Clifford multiplication):

\begin{itemize}
\item For all $X, Y \in C^{\infty}(M,TM)$ and $\psi \in C^{\infty}(M,\Sigma_{g} M)$ we have $X\cdot Y \cdot \psi +Y\cdot X \cdot \psi=-2g(X,Y)\psi$. 
\item If $(\cdot,\cdot)$ denotes the Hermitian metric on $\Sigma_{g} M$, then for all $X \in C^{\infty}(M,TM)$ and $\psi, \phi \in C^{\infty}(M,\Sigma_{g} M)$ we have $(X\cdot \psi , \phi)=-(\psi, X\cdot \phi)$.
\item For all $\psi, \phi \in C^{\infty}(M,\Sigma_{g} M)$ and $X\in C^{\infty}(M,TM)$, then $X(\psi,\phi)=(\nabla_{X}\psi,\phi)+(\psi,\nabla_{X} \phi)$.
\item For all $X, Y \in C^{\infty}(M,TM)$ and $\psi \in C^{\infty}(M,\Sigma_{g} M)$ we have $\nabla_{X} (Y\cdot \psi)=(\nabla_{X}Y)\cdot \psi+Y\cdot \nabla_{X}\psi$.
\end{itemize}

\noindent
For the rest of the paper, we let $\langle \cdot, \cdot \rangle:= Re (\cdot,\cdot)$. Then $\langle \cdot, \cdot \rangle$ defines a metric on $\Sigma_{g} M$.
The Dirac operator $D_{g}$ is then defined on $C^{\infty}(M,\Sigma_{g} M)$ as the composition of the Clifford multiplication and the connection $\nabla$. Indeed, if $(e_{1}, \cdots, e_{n})$ is a local orthonormal frame around a point $p\in M$ and $\psi \in C^{\infty}(M,\Sigma_{g} M)$ then one can locally define $D_{g}$ by  
$$D_{g}\psi :=\sum_{i=1}^{n}e_{i}\cdot \nabla_{e_{i}} \psi.$$
The Dirac operator is a natural first order operator acting on smooth sections of $\Sigma_{g} M$. Moreover, if $M$ is compact, then $D_{g}$ is essentially self-adjoint on $L^{2}(\Sigma_{g} M):=L^{2}(M,\Sigma_{g} M)$, with compact resolvent. In particular, there exists a complete orthonormal basis $(\varphi_{k})_{k\in \Z}$ of $L^{2}(\Sigma_{g} M)$ consisting of eigenspinors of $D_{g}$. That is $D_{g}\varphi_{k}=\lambda_{k}\varphi_{k}$, with $\lambda_{k}\to \pm \infty$ when $k\to \pm \infty$. We will use the convention that $\lambda_{k}>0$ (resp. $\lambda_{k}<0$) when $k>0$ (resp. $k<0)$.

\begin{proposition}[\cite{Bour, Friedrich2000}]
Consider a compact spin manifold $(M,g,\Sigma_{g}M)$, then
\begin{itemize}
\item[i)] The Dirac operator $D_{g}$ is conformally invariant. That is, if $\hat{g}:=e^{2u}g$, then there exists a unitary isomorphism $F_{g,\hat{g}}:\Sigma_{g}M \to \Sigma_{\hat{g}}M$ so that for $\varphi \in C^{\infty}(M,\Sigma_{g}M)$, $$D_{\hat{g}}(e^{-\frac{n-1}{2}u}F_{g,\hat{g}}(\varphi))=e^{-\frac{n+1}{2}u}F_{g,\hat{g}}(D_{g}\varphi).$$
\item[ii)] For $\varphi \in C^{\infty}(M,\Sigma_{g}M)$, $D_{g}^{2}\varphi =-\Delta_{g}\varphi +\frac{R_{g}}{4}\varphi$, where $R_{g}$ is the scalar curvature. 
\end{itemize}
\end{proposition}

\noindent
In what follows, we will identify spinors $\varphi \in \Sigma_{g}M$ with their isomorphic image $F_{g,\hat{g}}(\varphi)$, unless there is a specific distinction. Notice that as a result of the previous Proposition we have that $D_{g}$ is invertible if the Yamabe invariant of $(M,g)$ is positive: in fact, by the conformal invariance of the Dirac operator, point $i)$, the vanishing of the kernel is preserved by a conformal change, therefore if the Yamabe constant is positive, then
the conformal class of the metric contains a metric with positive scalar curvature, hence point $ii)$ implies the vanishing of the kernel of the Dirac operator on this metric.\\ 
We can define now the (unbounded) operator $|D_{g}|^{s}:L^{2}(\Sigma_{g}M)\to L^{2}(\Sigma_{g}M)$, for $s>0$ by $$|D_{g}|^{s}\psi =\sum_{k\in \Z} |\lambda_{k}|^{s}a_{k}\varphi_{k},$$
for $\psi=\sum_{k\in \Z}a_{k}\varphi_{k}$. The Sobolev space $H^{\frac{1}{2}}(\Sigma_{g} M)$ is then defined by
$$H^{\frac{1}{2}}(\Sigma_{g} M):=\{\psi \in L^{2}(\Sigma_{g} M); |D_{g}|^{\frac{1}{2}}\psi \in L^{2}(\Sigma_{g} M)\}.$$
This function space is equivalent to the classical $H^{\frac{1}{2}}$-Sobolev space and will be endowed with the inner product $\langle \cdot, \cdot \rangle_{\frac{1}{2}}$ defined by
$$\langle \psi,\phi \rangle_{\frac{1}{2}}:=\int_{M}\langle |D_{g}|^{\frac{1}{2}}\psi,|D_{g}|^{\frac{1}{2}}\phi \rangle \ dv_{g}, \forall \psi, \phi \in H^{\frac{1}{2}}(\Sigma_{g} M).$$
Notice that this inner product defines a natural semi-norm on $H^{\frac{1}{2}}(\Sigma_{g} M)$ by setting
$$\|\psi \|_{\frac{1}{2}}:=\||D_{g}|^{\frac{1}{2}}\psi\|_{L^{2}}.$$
This semi-norm becomes a norm when $D_{g}$ is invertible. Using the spectral resolution of $D_{g}$, we can split the space $H^{\frac{1}{2}}(\Sigma_{g} M)$ in a convenient way that fits our analysis. That is, we can write 
\begin{equation}\label{split}
H^{\frac{1}{2}}(\Sigma_{g} M)=H^{-}\oplus H^{0}\oplus H^{+} \; ,
\end{equation}
with
$$H^{-}:=\overline{\text{span}\{\varphi_i\}_{i<0}},\quad
H^{0}:=\ker D_{g}, \quad
H^{+}:=\overline{\text{span}\{\varphi_i\}_{i>0}}.$$
In particular, if the manifold $(M,g)$ has a positive Yamabe invariant, then $H^{0}=\{0\}$; therefore, we have
$$H^{\frac{1}{2}}(\Sigma_{g} M)=H^{-}\oplus H^{+} \; .$$
We will denote by $P^{-}$ (resp. $P^{+}$) the $L^{2}$-orthogonal projection onto $H^{-}$ (resp. $H^{+}$). We finish this section with the following proposition providing an expansion of the Green's function of the conformal Laplacian $L_{g}$ (Lemma 6.4 in \cite{LP}).

\begin{proposition}\label{propgreen}
We consider a compact Riemannian manifold $(M,g)$ as above with positive Yamabe constant $Y(M,[g])>0$, then the Green's function $G_{g}$ of the conformal Laplacian $L_{g}$ has the following asymptotic expansion, in a conformal normal coordinate system centered at a point $p$, that we can identify with the origin, with $r=|x|$.
\begin{itemize}
\item[i)] For $n=3,4,5$ or $(M,g)$ locally conformally flat around a point $p$: 
$$G_{g}(p,x)=b_{n}\Big(\frac{1}{r^{n-2}}+A(p)\Big)+o(1).$$
\item[ii)] For $n=6$
$$G_{g}(p,x)=b_{n}\Big(\frac{1}{r^{4}}-\frac{1}{1440}|W(p)|^{2}\ln(r)\Big)+O(1).$$
\item[iii)] For $n\geq 7$
$$G_{g}(p,x)=\frac{b_{n}}{r^{n-2}}\Big(1+\frac{(n-2)}{48(n-1)(n-4)}\Big(\frac{r^{4}}{12(n-6)}|W(p)|^{2}-R_{,ij}x^{i}x^{j}r^{2}\Big)\Big)+O(r^{7-n}).$$
where $b_{n}=\frac{1}{(n-2)\omega_{n-1}}$, $W(p)$ is the Weyl tensor at $p$ and $R_{,ij}$ is the covariant derivative of the scalar curvature $R=R_g$ with respect to $\partial_{x_{i}}$ and $\partial_{x_{j}}$.
\end{itemize}
\end{proposition}

\noindent
The term $A(p)$ that appears in $i)$, in the Proposition above, is traditionally called the mass (since it is up to a multiplicative constant, equal to the Arnowitt-Deser-Misner (ADM) mass of the blown-up manifold, namely the stereographic projection of $M$ from $p$ using the Green's function of the conformal Laplacian). The sign of $A$ is very important and it has a deep meaning related to the positive mass theorem \cite{S,W} (we also refer the reader to \cite{LP} for a more detailed approach from the perspective of the Yamabe problem). Indeed, in the case of positive Yamabe invariant, one has $A(p)\geq 0$ and $A(p)=0$ if and only if $(M,[g])$ is conformal to the standard sphere $(S^{n},[g_{0}])$.

\section{Functional Setting}

\noindent
We recall here the expression of the energy functional (\ref{modifiedJg}) that we are dealing with
$$ J_{g}(\psi)=\frac{1}{2}\int_{M}\langle D \psi,\psi\rangle   \ dv_{g} -\frac{1}{4}\int_{M\times M} V_{g}(x,y)|\psi(x)|^{2} |\psi(y)|^{2} \ d_{g}v(x) dv_{g}(y) \; .$$
For $\psi\in H^{\frac{1}{2}}$ we will write $\psi = \psi^{+}+\psi^{-} \in H^{+}\oplus H^{-}$. We will state below the main tools that would allow us to estimate the ground state energy of the functional. This was proven in \cite{MMM} following the process introduced in \cite{YX1}.
\begin{proposition}\label{proptau}
There exits a $C^{1}$-map $\tau: H^{+}\to H^{-}$ such that for every $\psi \in H^{+}$ 
$$J_{g}(\psi+h)<J_{g}(\psi+\tau(\psi)), \forall h\in H^{-} , \; h\not=\tau(h).$$
Moreover, $\tau$ satisfies the following properties:
\begin{itemize}
\item[i)] $P^{-}\Big[D_{g}\tau(\psi)-\left(\displaystyle\int_{M}V_{g}(\cdot,y)|\tau(\psi)(y)+\psi(y)|^{2}\ dv(y) \right)(\psi+\tau(\psi))\Big]=0$.
\item[ii)] $\|\tau(\psi)\|_{\frac{1}{2}}^{2}\leq \displaystyle\frac{1}{2} \int_{M\times M}V_{g}(x,y) |\psi(x)|^{2} |\psi(y)|^{2}\ dv(y)dv(x)$.
\item[iii)] If $K(\psi):=\displaystyle\frac{1}{4} \int_{M\times M} V_{g}(x,y) |\psi(x)|^{2} |\psi(y)|^{2} \ dv(x)dv(y)$, then (here $\|\cdot\|_{Op}$ is the operator norm)
 $$\|\nabla \tau(\psi)\|_{Op}\leq \|\nabla^{2}K(\psi+\tau(\psi))\|_{Op}.$$
\item[iv)] Let $\tilde{J}:H^{+}\to \R$ defined by $\tilde{J}(\psi):=J_{g}(\psi+\tau(\psi))$. If $(\psi_{k})_{k}$ is a (PS) sequence of $\tilde{J}$, then $(\psi_{k}+\tau(\psi_{k}))_{k}$ is a (PS)-sequence for $J_{g}$ and
$$\|\nabla J_{g}(\psi)\|=\|\nabla J_{g}(\psi+\tau(\psi))\|, \forall \psi \in H^{+}.$$
\end{itemize}
\end{proposition}

\noindent
As pointed out in \cite{MMM}, the modified functional $\tilde{J}$ has a mountain-pass geometry since $\tilde{J}(0)=J_{g}(0)=0$ and for $\psi \in H^{+}$ with $\|\psi\|_{\frac{1}{2}}=1$, we have
$$\tilde{J}(t\psi)\geq J_{g}(t\psi)=\frac{t^{2}}{2}-\frac{t^{4}}{4}\int_{M} (V_{g}*|\psi|^{2})|\psi|^{2}\ dv.$$
In particular, there exists $t_{0}>0$ and $\nu_{0}>0$ such that 
$$\tilde{J}(t\psi)\geq 0, \forall \;  0\leq t\leq t_{0} \text{ and } \tilde{J}(t_{0}\psi)\geq \nu_{0}.$$
 This structure allows us to find a candidate critical level for the functional $\tilde{J}$ and hence a critical point for $J_{g}$, via the following min-max scheme:
$$\delta_{0}:=\inf_{\psi \in H^{+}\setminus\{0\}}\max_{t>0}\tilde{J}(t\psi).$$
Notice that $\delta_{0}$ is well-defined and $\delta_{0}\geq \nu_{0}>0$. Moreover, it corresponds to the ground-state solutions of $\tilde{J}_{g}$ (if they exist). This ground state can be further characterized on the Nehari manifold
$$\mathcal{M}=\left\{\psi \in H^{+}; \langle \nabla \tilde{J}(\psi),\psi\rangle=0 \right\}; $$
then
\begin{equation}\label{delta0}
\delta_{0}=\inf_{\psi\in \mathcal{M}} \tilde{J}(\psi) \; ,
\end{equation}
as long as $\mathcal{M}\not=\emptyset$. Hence, in order to obtain genuine critical points for $J_{g}$ via variational analysis, we need that $\mathcal{M}\not=\emptyset$ and $\tilde{J}$ satisfies the (PS) condition at the min-max level $\delta_{0}$. This will allow us to state that $\delta_{0}$ is indeed a critical value for $\tilde{J}$ and hence for $J_{g}$. On the other hand, we know from Theorem \ref{thm1} and Corollary \ref{cor1} that the (PS) condition for $\tilde{J}$ and $J_{g}$ is satisfied below the energy level $\overline{Y}(S^{n},[g_{0}])$. Hence we can achieve our objective if we can show that $\delta_{0}<\overline{Y}(S^{n},[g_{0}])$.

\noindent
In the case of the Yamabe problem, the process consists of grafting a standard bubble (Aubin-Talenti test function) and then expand its energy. But in our case, there are some additional subtleties that need to be addressed. Indeed, we are going to apply the process for the functional $\tilde{J}$, then the $\tau$-component of any potential test spinor needs to be controlled. Thus, as it was the case in \cite{MMM}, we will estimate the energy level of a (PS) sequence of $\tilde{J}$ in terms of the energy levels of $J_{g}$.\\
Therefore, we consider a $(PS)_{c}$ sequence $(\psi)_{k}$ for $J_{g}$, namely
$$J_{g}(\psi_{k})\to c>0, \quad \|\nabla J_{g}(\psi_{k})\|\to 0 \; .$$
One can easily verify that $\|\psi_{k}\|_{\frac{1}{2}}$ is bounded; moreover, we have the following properties (\cite{MMM}):
\begin{proposition}\label{prop diff}
Assume that $(\psi_{k})_{k}$ is a $(PS)_{c}$ sequence for $J_{g}$ with $c>0$. We have
\begin{itemize}
\item[i)] $\|\psi_{k}^{-}-\tau(\psi_{k})\|_{\frac{1}{2}}=O\Big(\|\nabla J_{g}(\psi_{k})\|\Big)$.
\item[ii)]$\nabla\tilde{J}(\psi_{k}^{+})\to 0$.
\item[iii)] There exists $t_{k}>0$ such that $t_{k}\psi_{k}^{+}\in \mathcal{M}$. Moreover, $|t_{k}-1|=O\Big(\|\nabla \tilde{J}(\psi_{k}^{+})\|\Big)$.
\end{itemize}
\end{proposition}

\noindent
This last Proposition tells us in particular that if we have a $(PS)_{c}$ sequence for $J_{g}$ then $\mathcal{M}\not= \emptyset$ and hence $\delta_{0}$ is well-defined. Next, we can exhibit an upper-bound for $\delta_{0}$ in terms of the energy of the 
$(PS)_{c}$ sequence $(\psi_{k})_{k}$ (\cite{MMM})

\begin{proposition}\label{prop exp}
Assume that $(\psi_{k})_{k}$ is a $(PS)_{c}$ sequence for $J_{g}$ with $c>0$. Then
$$\delta_{0}\leq J_{g}(\psi_{k})+O\Big(\|\nabla J_{g}(\psi_{k})\|^{2}\Big).$$
In particular, if $J_{g}$ satisfies the (PS) condition at the level $\delta_{0}$ then it has a critical point $\psi$ at that level.
\end{proposition}

\noindent
It is clear now that in order to have a precise upper-bound for $\delta_{0}$, we need a parameterized test function that behaves as a $(PS)_{c}$ sequence and such that the decay of its gradient is fast enough so that it does not interfere with the energy expansion. This last part is not present in the classical Yamabe problem and that is because the minimization is more direct and does not go through an alternative functional $\tilde{J}$.

\section{Test Spinor and Energy Estimates}

\noindent
In this section we will start by constructing an adequate test function and then estimating its gradient and energy. The test function that we will be using is similar to the one introduced in \cite{MMM}. But compared to \cite{MMM} a more precise expansion is required. Notice that in \cite{MMM}, we only needed to expand up to first order to have a signed quantity allowing us to go below the critical level and that is because of the presence of the quadratic Brezis-Nirenberg term in the functional. In \cite{YX1}, the authors needed to go up to order 4 in order to obtain a signed term. This will be our objective when the Weyl tensor does not vanish. In the locally conformally flat case, we need to expand to an even higher order. It is crucial here to keep track of the growth of the gradient term of $J_{g}$ so that it does not interfere with the signed quantities that we need. The construction of the test spinor is very similar to the one in \cite{MMM,YX1}, but for the sake of completion, we will write here its explicit form.

\noindent
Consider a constant spinor $\Psi_{0}$ on $\R^{n}$ such that $|\Psi_{0}|^{2}=a_{n}$, where $a_{n}$ is  a constant satisfying
$$a_{n}^{\frac{n}{n-1}}\omega_{n}=2^{n}\overline{Y}(S^{n},[g_{0}])c_{n}^{-\frac{1}{n-1}},$$
and 
\begin{equation}\label{cn}
c_{n}=\frac{\lambda^{+}(S^{n}, [g_{0}])^{n-2}}{ Y_{\frac{n}{2}-1}(S^{n}, [g_{0}])}.
\end{equation}
We define then the spinor 
\begin{equation}\label{psi}
\Psi=\Big(\frac{1}{1+|x|^{2}}\Big)^{\frac{n}{2}}(1-x)\cdot \Psi_{0},
\end{equation}
so that if $f(r)=\frac{1}{1+r^{2}}$, then $|\Psi(x)|^{2}=a_{n}\big(f(|x|)\big)^{n-1}$. Notice that 
$$D_{\R^{n}}\Psi=\frac{n}{2}f\Psi.$$
The spinor $\Psi$ is the standard bubble that appear in Theorem \ref{thm1}. It is also the same Spinor that appear in the spinorial Yamabe problem \cite{YX1}. We fix $\delta>0$ so that $2\delta<i(M)$, the injectivity radius of $M$. We let $\eta$ to be a smooth function on $\R^{n}$ with support in $B_{2\delta}(0)=:B_{2\delta}$ such that $\eta=1$ on $B_{\delta}(0)=:B_{\delta}$. Now, we can define the spinor 
\begin{equation}\label{psiepsilon}
\psi_{\varepsilon}(x)=\eta(x)\varepsilon^{-\frac{n-1}{2}}\Psi\left(\frac{x}{\varepsilon}\right)=\eta(x) \Psi_{\varepsilon}(x).
\end{equation}
Next, we use the Bourguignon-Gauduchon \cite{Bour} trivialization in order to graft the spinor $\psi_{\varepsilon}$ on $M$. Indeed,  we fix $p_{0}\in M$ and $(x_{1},\cdots,x_{n})$ local normal coordinates around $p_{0}$ provided by the exponential map $\exp_{p_{0}}$. That is, there exists a neighborhood $U\subset T_{p_{0}}M=\R^{n}$ and a neighborhood $V\subset M$, such that $\exp_{p_{0}}:U\to V$ is a diffeomorphism.

\noindent
Let $G(p)=(g_{ij}(p))_{ij}$ be the metric at $p$ and $B=G^{-\frac{1}{2}}$. Notice that $B$ is well defined since $G$ is symmetric and positive definite. With these notations, we have that $B^{*}g=g_{\R^{n}}$. Therefore, $B$ defines an isometry as a map $B(p):(T_{\exp^{-1}_{p_{0}}p}U,g_{\R^{n}})\to (T_{p}V,g(p))$. Hence, given an oriented frame $(y_{1},\cdots, y_{n})$ on $U$, we obtain a natural oriented frame on $V$ by taking $(By_{1},\cdots, B y_{n})$. Thus, one has an isomorphism of the $SO(n)$-principal bundle induced by the map $\Phi(y_{1},\cdots, y_{n})=(By_{1},\cdots, B y_{n})$ as described in the diagram below:

\begin{center}

\tikzset{every picture/.style={line width=0.75pt}} 

\begin{tikzpicture}[x=0.75pt,y=0.75pt,yscale=-1,xscale=1]

\draw    (211,120) -- (376,120.99) ;
\draw [shift={(378,121)}, rotate = 180.34] [color={rgb, 255:red, 0; green, 0; blue, 0 }  ][line width=0.75]    (10.93,-3.29) .. controls (6.95,-1.4) and (3.31,-0.3) .. (0,0) .. controls (3.31,0.3) and (6.95,1.4) .. (10.93,3.29)   ;
\draw    (212,212) -- (377,212.99) ;
\draw [shift={(379,213)}, rotate = 180.34] [color={rgb, 255:red, 0; green, 0; blue, 0 }  ][line width=0.75]    (10.93,-3.29) .. controls (6.95,-1.4) and (3.31,-0.3) .. (0,0) .. controls (3.31,0.3) and (6.95,1.4) .. (10.93,3.29)   ;
\draw    (182,133.5) -- (182.97,193.5) ;
\draw [shift={(183,195.5)}, rotate = 269.08] [color={rgb, 255:red, 0; green, 0; blue, 0 }  ][line width=0.75]    (10.93,-3.29) .. controls (6.95,-1.4) and (3.31,-0.3) .. (0,0) .. controls (3.31,0.3) and (6.95,1.4) .. (10.93,3.29)   ;
\draw    (395,134.5) -- (395.97,194.5) ;
\draw [shift={(396,196.5)}, rotate = 269.08] [color={rgb, 255:red, 0; green, 0; blue, 0 }  ][line width=0.75]    (10.93,-3.29) .. controls (6.95,-1.4) and (3.31,-0.3) .. (0,0) .. controls (3.31,0.3) and (6.95,1.4) .. (10.93,3.29)   ;

\draw (117,102) node [anchor=north west][inner sep=0.75pt]    {$P_{SO}( U,g_{\mathbb{R}^{n}})$};
\draw (384,102) node [anchor=north west][inner sep=0.75pt]    {$P_{SO}( V,g) \subset P_{SO}( M,g)$};
\draw (127,200) node [anchor=north west][inner sep=0.75pt]    {$U\subset T_{p_{0}} M$};
\draw (385,202) node [anchor=north west][inner sep=0.75pt]    {$V\subset M$};
\draw (281,94) node [anchor=north west][inner sep=0.75pt]    {$\Phi $};
\draw (268,181) node [anchor=north west][inner sep=0.75pt]    {$\exp_{p_{0}}$};

\end{tikzpicture}

\end{center}

\noindent
The map $\Phi$ commutes with the right action of $SO(n)$ and hence it induces an isomorphism of spin structures:

\begin{center}
\tikzset{every picture/.style={line width=0.75pt}} 

\begin{tikzpicture}[x=0.75pt,y=0.75pt,yscale=-1,xscale=1]

\draw    (211,120) -- (376,120.99) ;
\draw [shift={(378,121)}, rotate = 180.34] [color={rgb, 255:red, 0; green, 0; blue, 0 }  ][line width=0.75]    (10.93,-3.29) .. controls (6.95,-1.4) and (3.31,-0.3) .. (0,0) .. controls (3.31,0.3) and (6.95,1.4) .. (10.93,3.29)   ;
\draw    (163,211) -- (422,211.99) ;
\draw [shift={(424,212)}, rotate = 180.22] [color={rgb, 255:red, 0; green, 0; blue, 0 }  ][line width=0.75]    (10.93,-3.29) .. controls (6.95,-1.4) and (3.31,-0.3) .. (0,0) .. controls (3.31,0.3) and (6.95,1.4) .. (10.93,3.29)   ;
\draw    (105,128.5) -- (105.97,188.5) ;
\draw [shift={(106,190.5)}, rotate = 269.08] [color={rgb, 255:red, 0; green, 0; blue, 0 }  ][line width=0.75]    (10.93,-3.29) .. controls (6.95,-1.4) and (3.31,-0.3) .. (0,0) .. controls (3.31,0.3) and (6.95,1.4) .. (10.93,3.29)   ;
\draw    (469,132.5) -- (469.97,192.5) ;
\draw [shift={(470,194.5)}, rotate = 269.08] [color={rgb, 255:red, 0; green, 0; blue, 0 }  ][line width=0.75]    (10.93,-3.29) .. controls (6.95,-1.4) and (3.31,-0.3) .. (0,0) .. controls (3.31,0.3) and (6.95,1.4) .. (10.93,3.29)   ;

\draw (7,102) node [anchor=north west][inner sep=0.75pt]    {$U\times Spin( n) =P_{Spin}( U,g_{\mathbb{R}^{n}})$};
\draw (384,102) node [anchor=north west][inner sep=0.75pt]    {$P_{Spin}( V,g) \subset P_{Spin}( M,g)$};
\draw (75,200) node [anchor=north west][inner sep=0.75pt]    {$U\subset T_{p_{0}} M$};
\draw (431,199) node [anchor=north west][inner sep=0.75pt]    {$V\subset M$};
\draw (275,89) node [anchor=north west][inner sep=0.75pt]    {$\tilde{\Phi }$};
\draw (268,181) node [anchor=north west][inner sep=0.75pt]    {$\exp_{p_{0}}$};

\end{tikzpicture}
\end{center}

\noindent
This leads to an isomorphism between the spin bundles $\Sigma_{g_{\R^{n}}} U$ and $\Sigma_{g} V$. If we let $e_{i}=B(\partial_{x_{i}})$ we then obtain an orthonormal frame $(e_{1},\cdots e_{n})$ of $(TV,g)$. We let $\nabla$ and $\overline{\nabla}$, respectively the Levi-Civita connections on $(TU,g_{\R^{n}})$ and $(TV,g)$. We will keep the same notations for their natural lifts to $\Sigma_{g_{\R^{n}}} U$ and $\Sigma_{g} V$. From now on, if $H\to U$ (resp. $H\to V$) is a smooth bundle over $U$ (resp. over $V$), we let $\Gamma(H)$ be the space of smooth sections of $H$. The Clifford multiplications then satisfy
$$e_{i}\cdot \overline{\psi}=B(\partial_{x_{i}})\cdot \overline{\psi}=\overline{\partial_{x_{i}}\cdot \psi},$$
where here we use the identification that any $\psi \in \Gamma (\Sigma_{g_{\R^{n}}} U)$ corresponds via the previously defined isomorphism to a spinor $\overline{\psi}\in \Gamma(\Sigma_{g} V)$. If $D$ and $\overline{D}$ are the Dirac operators acting on $\Gamma (\Sigma_{g_{\R^{n}}} U)$ and $\Gamma (\Sigma_{g} V)$, then we have for $\psi \in \Gamma(\Sigma_{g_{\R^{n}}} U)$ 
$$\overline{D} \; \overline{\psi}=\overline{D\psi} +\Omega\cdot \overline{\psi}+X\cdot \overline{\psi}+\sum_{i,j}(b_{ij}-\delta_{ij})\overline{\partial_{x_{i}}\cdot \nabla_{\partial_{x_{j}}}\psi},$$
where here, the $b_{ij}$ are such that $e_{i}=\sum_{j} b_{ij}\partial_{x_{j}}$, $\Omega\in \Gamma(Cl(TV))$ and $X\in \Gamma(TV)$ are defined by
$$\Omega=\frac{1}{4} \sum_{i,j,k;\\ i\not=j\not=k\not=i} \sum_{\alpha,\beta} b_{i\alpha}(\partial_{x_{\alpha}}b_{j\beta})b^{-1}_{\beta k}e_{i}\cdot e_{j}\cdot e_{k},$$
and
$$X=\frac{1}{4}\sum_{i,k}(\overline{\Gamma}_{ik}^{i}-\overline{\Gamma}_{ii}^{k})e_{k}=\frac{1}{2}\sum_{i,k}\overline{\Gamma}_{ik}^{i}e_{k}.$$
Using the identification between $x\in \R^{n}$ and $p=\exp_{p_{0}}x\in M$, we have the following asymptotic expansion in normal coordinates, proved in \cite{YX1}:

\begin{proposition}\label{propest1}
In local coordinates we have
\begin{itemize}
\item[i)] $b_{ij}-\delta_{ij}=-\frac{1}{6}R_{i\alpha\beta j}x^{\alpha}x^{\beta}-\frac{1}{12}R_{i\alpha\beta j,k}x^{\alpha}x^{\beta}x^{k}+O(r^{4})$.
\item[ii)] $X=-\sum\limits_{k=1}^{n}\Big(\frac{1}{4}R_{\alpha k}x^{\alpha}+\frac{1}{6}R_{\alpha k,\beta}x^{\alpha}x^{\beta}+O(r^{3})\Big)e_{k}$.
\item[iii)] $\Omega=-\frac{1}{144}\sum\limits_{\begin{array}{cc}
i,j,k\\
 i\not=j\not=k\not=i
\end{array}}\sum\limits_{\ell}R_{\ell \beta \gamma k}(R_{ji\alpha\ell}+R_{j\ell\alpha i})x^{\alpha}x^{\beta}x^{\gamma}e_{i}\cdot e_{j}\cdot e_{k}+O(r^{4})$.
\end{itemize}
\end{proposition}

\noindent
These asymptotics can be improved in an adequately chosen coordinate chart, namely, in conformal normal coordinates (\cite{LP} and \cite{YX1}).

\begin{proposition}\label{proplp} {\textcolor{red}{ 22/9; tau 1/9 }}
Given $p\in M$ and $N\geq 2$. Then there exits a conformal metric $g$ on $M$ such that around the point $p$ we have
$$det(g_{ij})=1+O(r^{N}).$$
Moreover, at the point $p$, it holds
\begin{itemize}
\item[i)] $R_{ij}=0$.
\item[ii)] $R_{ij,k}+R_{jk,i}+R_{ki,j}=0$.
\item[iii)] $\Big(R_{\alpha \beta,k\ell}+\frac{2}{9}\sum \limits_{i,d}R_{i\alpha \beta d}R_{ik\ell d}\Big)x^{\alpha}x^{\beta}x^{k}x^{\ell}=0.$
\item[iv)] $-\Delta R=-R_{,kk}=\frac{1}{6}|W|^{2}$.
\end{itemize}
\end{proposition}

\noindent
Our spinor $\Psi$ defined in $(\ref{psi})$, enjoys some nice properties in these coordinates (\cite{YX1}).

\begin{lemma}\label{lemyx}
For $\Psi$ defined as in $(\ref{psi})$, we have for any $x\in \R^n$
$$R_{i\alpha\beta j}x^{\alpha}x^{\beta}\partial_{x_{i}}\cdot \nabla _{\partial_{x_{j}}}\Psi (x) =0.$$
Moreover, if $A_{ijk\ell}\in \R$ with $1\leq i,j,k,\ell \leq n$, then we can choose $\Psi_{0}$ in $(\ref{psi})$, so that
$$A_{ijk\ell}\langle \partial_{x_{i}}\cdot \partial_{x_{j}}\cdot \partial_{x_{k}}\cdot \partial_{x_{\ell}}\cdot \Psi_{0},\Psi_{0}\rangle =0.$$
\end{lemma}

\noindent
We can now define our test spinor $\varphi_{\varepsilon}:=\overline{\psi_{\varepsilon}}$, with $\psi_{\varepsilon}$ defined in (\ref{psiepsilon}). We recall that our goal in here is to apply Proposition \ref{prop exp} for the test spinor $\varphi_{\varepsilon}$.  In order to do that, we need to show that $(\varphi_{\varepsilon})_{\varepsilon}$ is a $(PS)_{c}$ sequence for $J_{g}$. Notice that, only the convergence of $\nabla J_{g}(\varphi_{\varepsilon})$ to zero will not be enough. Instead, we need a precise estimate on the convergence rate to zero in terms of $\varepsilon$. In all what follows, $C$ will denote a suitable positive constant independent of $\varepsilon$.

\begin{proposition}\label{lemma1 exp}
For $\varphi_{\varepsilon}$ defined as above, we have
$$\|\nabla J_{g}(\varphi_{\varepsilon})\|_{H^{-\frac{1}{2}}}\leq \left\{\begin{array}{ll}
O(\varepsilon^{\frac{n-1}{2}}), &  4\leq n\leq 6\\
\\
O(\varepsilon^{3}|\ln(\varepsilon)|^{\frac{4}{7}}), &  n=7\\
\\
O(\varepsilon^{3}), &  n\geq 8
\end{array}
\right. .$$
\end{proposition}

\begin{proof}
We will proceed by estimating first the $H^{-\frac{1}{2}}$-norm of $\varphi_{\varepsilon}$ and then the norm of 
$$\nabla J_{g}(\varphi_{\varepsilon})=\overline{D}\varphi_{\varepsilon}-\Big(V_{g}*|\varphi_{\varepsilon}|^{2}\Big)\varphi_{\varepsilon}\; .$$ 
Here, $H^{-\frac{1}{2}}$ is the dual of the space $H^{\frac{1}{2}}(M)$ equipped with the norm $\|\cdot \|_{\frac{1}{2}}$. By the continuous embedding $L^{\frac{2n}{n+1}}(M)\hookrightarrow H^{-\frac{1}{2}}(M)$, we have

\begin{align}
\|\varphi_{\varepsilon}\|_{H^{-\frac{1}{2}}}&\leq C \|\varphi_{\varepsilon}\|_{L^{\frac{2n}{n+1}}}=\Big(\int_{B_{2\delta}}|\varphi_{\varepsilon}|^{\frac{2n}{n+1}}\ dv\Big)^{\frac{n+1}{2n}}\notag\\
&\leq C \Big(\int_{|x|\leq 2\delta} |\psi_{\varepsilon}(x)|^{\frac{2n}{n+1}}\ dx \Big)^{\frac{n+1}{2n}}\notag\\
&\leq C\varepsilon\Big(\int_{0}^{\frac{2\delta}{\varepsilon}} \frac{r^{n-1}}{(1+r^{2})^{\frac{n(n-1)}{n+1}}}\ dr \Big)^{\frac{n+1}{2n}} \notag\\
&\leq C\left\{\begin{array}{ll}
\varepsilon |\ln(\varepsilon)|^{\frac{2}{3}}, & n=3\\
\\
\varepsilon, & n\geq 4 
\end{array}
\right..\label{phieps}
\end{align}
Next, we focus on estimating $\nabla J_{g}(\varphi_{\varepsilon})$. Indeed, we have
\begin{align}
\overline{D}\varphi_{\varepsilon}&=\overline{D\psi_{\varepsilon}}+\Omega\cdot \overline{\psi_{\varepsilon}}+X\cdot \overline{\psi}+\sum_{i,j}(b_{ij}-\delta_{ij})\overline{\partial_{x_{i}}\cdot \nabla_{\partial_{x_{j}}}\psi_{\varepsilon}}\notag\\
&=\Big(\int_{\R^{n}}V_{g_{\R^{n}}}(x,y)|\Psi_{\varepsilon}(y)|^{2}\ dy\Big) \varphi_{\varepsilon}+(\nabla\eta(x)+X) \cdot \varphi_{\varepsilon}+\Omega\cdot \varphi_{\varepsilon}\notag\\
&\quad+\sum_{i,j}(b_{ij}-\delta_{ij})\overline{\partial_{x_{i}}\cdot \nabla_{\partial_{x_{j}}}\psi_{\varepsilon}}.\label{decomp}
\end{align}

\noindent
On the other hand,
\begin{align}
\overline{D \psi_{\varepsilon}}-(V_{g}*|\varphi_{\varepsilon}|^{2})\varphi_{\varepsilon}&=\Big(\int_{\R^{n}}V_{g_{\R^{n}}}(x,y)|\Psi_{\varepsilon}(y)|^{2}\ dy \Big)\varphi_{\varepsilon} \notag\\ 
&\quad - \Big(\int_{M}V_{g}(x,y)|\varphi_{\varepsilon}|^{2}\ dv(y)\Big) \varphi_{\varepsilon}+\nabla\eta(x) \cdot \varphi_{\varepsilon}\notag\\
&=\Big(\int_{|x-y|<\frac{\delta}{2}}[V_{g_{\R^{n}}}(x,y)-V_{g}(x,y)]|\varphi_{\varepsilon}(y)|^{2}\ dy \Big) \varphi_{\varepsilon}\notag\\
&\quad + \Big(\int_{|x-y|<\frac{\delta}{2}}V_{g}(x,y)O(|y|^{N})|\varphi_{\varepsilon}(y)|^{2}dy\Big)\varphi_{\varepsilon}\notag\\
&\quad + \Big(\int_{|x-y|>\frac{\delta}{2}}V_{g_{\R^{n}}}(x,y)|\Psi_{\varepsilon}(y)|^{2} dy\Big) \varphi_{\varepsilon}\notag\\
&\quad - \Big(\int_{|x-y|>\frac{\delta}{2}}V_{g}(x,y)|\varphi_{\varepsilon}|^{2}\ dv(y) \Big) \varphi_{\varepsilon} +\nabla \eta \cdot \varphi_{\varepsilon}.\notag
\end{align}

\noindent
This leads to 
\begin{align}
\overline{D}\varphi_{\varepsilon}-(V_{g}*|\varphi_{\varepsilon}|^{2})\varphi_{\varepsilon}&=\Big(\int_{|x-y|<\frac{\delta}{2}}[V_{g_{\R^{n}}}(x,y)-V_{g}(x,y)]|\varphi_{\varepsilon}(y)|^{2}\ dy \Big) \varphi_{\varepsilon}\notag\\
&\quad+\Big(\int_{|x-y|>\frac{\delta}{2}}V_{g_{\R^{n}}}(x,y)|\Psi_{\varepsilon}(y)|^{2} dy\Big) \varphi_{\varepsilon}-\Big(\int_{|x-y|>\frac{\delta}{2}}V_{g}(x,y)|\varphi_{\varepsilon}|^{2}\ dv(y) \Big) \varphi_{\varepsilon} \notag\\
&\quad+\Big(\int_{|x-y|<\frac{\delta}{2}}V_{g}(x,y)O(|x-y|^{N})|\varphi_{\varepsilon}(y)|^{2}dy\Big)\varphi_{\varepsilon}\notag\\
&\quad +\nabla \eta \cdot \varphi_{\varepsilon} + \Omega\cdot \overline{\psi_{\varepsilon}}+X\cdot \overline{\psi}+\sum_{i,j}(b_{ij}-\delta_{ij})\overline{\partial_{x_{i}}\cdot \nabla_{\partial_{x_{j}}}\psi_{\varepsilon}}\notag\\
&=A_{1}+A_{2}+A_{3}+A_{4}+A_{5}+A_{6}+A_{7}+A_{8}.\notag
\end{align}
We will estimate now the terms $A_{i}$ for $i=1,\cdots,8$, in dimensions $n\geq 4$. Indeed, for $A_{5}$, we have
\begin{align}
\|A_{5}\|_{H^{-\frac{1}{2}}}&\leq C\|A_{5}\|_{L^{\frac{2n}{n+1}}}=C\Big(\int_{B_{2\delta}}|\overline{\nabla \eta \cdot \Psi_{\varepsilon}}|^{\frac{2n}{n+1}}\ dv \Big)^{\frac{n+1}{2n}}\notag\\
&\leq C\Big(\int_{\delta\leq |x|\leq 2\delta}|\Psi_{\varepsilon}(x)|^{\frac{2n}{n+1}}\ dx \Big)^{\frac{n+1}{2n}}\notag\\
&\leq C \varepsilon \Big(\int_{\frac{\delta}{\varepsilon}}^{\frac{2\delta}{\varepsilon}}\frac{r^{n-1}}{(1+r^{2})^{\frac{n(n-1)}{(n+1)}}}\ dr \Big)^{\frac{n+1}{2n}}\leq C \varepsilon^{\frac{n-1}{2}}.\notag
\end{align}
For $A_{1}$, we use Proposition \ref{propgreen}  in Section 2, in a small coordinate patch, we have 
$$|V_{g}(x,y)-V_{g_{\R^{n}}}(x,y)|=|h(x,y)|,$$ 
with $h$ having different behavior depending on the dimension $n$. Indeed,
\begin{itemize}
\item If $n\geq 7$ then $|h(x,y)|\leq Cr^{2}$. In this case, we have
\end{itemize}
\begin{align}
\|A_{1}\|_{H^{-\frac{1}{2}}}&\leq C\|A_{1}\|_{L^{\frac{2n}{n+1}}}\leq C\Big(\int_{B_{2\delta}}\Big(\int_{B_{2\delta}}h(x,y)|\Psi_{\varepsilon}(y)|^{2}\ dy |\Psi_{\varepsilon}(x)|\ \Big)^{\frac{2n}{n+1}}dx\Big)^{\frac{n+1}{2n}}\notag\\
&\leq C\Big(\int_{B_{2\delta}}\Big(\int_{B_{2\delta}}|y|^{2}|\Psi_{\varepsilon}(x-y)|^{2}\ dy |\Psi_{\varepsilon}(x)|\Big)^{\frac{2n}{n+1}}\ dx\Big)^{\frac{n+1}{2n}} \notag\\
&\leq  C \varepsilon^{3}.\notag
\end{align}
\begin{itemize}
\item If $n=6$, then we have $|h(x,y)|\leq Cr^{2}|\ln(r)|$. Thus,
\end{itemize}
\begin{align}
\|A_{1}\|_{H^{-\frac{1}{2}}}&\leq C\|A_{1}\|_{L^{\frac{2n}{n+1}}}\leq C\Big(\int_{B_{2\delta}}\Big(\int_{B_{2\delta}}|y|^{2}|\ln(|y|)||\Psi_{\varepsilon}(x-y)|^{2}\ dy |\Psi_{\varepsilon}(x)|\Big)^{\frac{2n}{n+1}}\ dx\Big)^{\frac{n+1}{2n}} \notag\\
&\leq  C \varepsilon^{3}|\ln(\varepsilon)|.\notag
\end{align}

\begin{itemize}
\item Finally, for $n=4$ or $5$, then $|h(x,y)|\leq C$. Therefore, a similar process as above yields 
\end{itemize}
$$\|A_{1}\|_{H^{-\frac{1}{2}}}\leq C \varepsilon^{2}.$$

\noindent
Notice that a similar inequality holds for $\|A_{4}\|_{H^{-\frac{1}{2}}}$, if we choose $N\geq 4$. Hence we move to $A_{2}$. We use the fact that the Green's function is bounded outside of the diagonal and we have
\begin{align}
\|A_{2}\|_{H^{-\frac{1}{2}}}&\leq C \Big(\int_{|x|<\frac{\delta}{4}}\Big(\int_{|x-y|>\frac{\delta}{2}}|\Psi_{\varepsilon}(y)|^{2}\ dy |\Psi_{\varepsilon}(x)|^{2}\Big)^{\frac{2n}{n+1}}\ dx \Big)^{\frac{n+1}{2n}}\notag\\
&+\Big(\int_{\frac{\delta}{4}<|x|<2\delta}\Big(\int_{|x-y|>\frac{\delta}{2}}|\Psi_{\varepsilon}(y)|^{2}\ dy |\Psi_{\varepsilon}(x)|^{2}\Big)^{\frac{2n}{n+1}}\ dx \Big)^{\frac{n+1}{2n}}\notag\\
&\leq C \varepsilon\Big( \|\varphi_{\varepsilon}\|_{L^{\frac{2n}{n+1}}}\int_{\frac{\delta}{4\varepsilon}}^{\infty}\frac{r^{n-1}}{(1+r^{2})^{n-1}}\ dr +\|\varphi_{\varepsilon}\|_{L^{2}}^{2}\Big(\int_{\frac{\delta}{4\varepsilon}}^{\frac{2\delta}{\varepsilon}} \frac{r^{n-1}}{(1+r^{2})^{\frac{n(n-1)}{n+1}}}\ dr \Big)^{\frac{n+1}{2n}}\Big)\notag\\
& \leq C \varepsilon^{\frac{n+1}{2}}.\notag
\end{align}
A similar inequality holds for $\|A_{3}\|_{H^{-\frac{1}{2}}}$. Now, using Proposition \ref{propest1}, we have $|\Omega|=O(|x|^{3})$. Therefore,
\begin{align}
\|A_{6}\|_{H^{-\frac{1}{2}}}&\leq C \Big(\int_{B_{2\delta}}|\Omega|^{\frac{2n}{n+1}}|\varphi_{\varepsilon}|^{\frac{2n}{n+1}}\ dv \Big)^{\frac{n+1}{n}} \leq C \Big( \int_{|x|\leq 2\delta} |x|^{\frac{6n}{n+1}}|\Psi_{\varepsilon}(x)|^{\frac{2n}{n+1}}\ dx\Big)^{\frac{n+1}{2n}} \notag\\
&\leq C \varepsilon^{4} \Big( \int_{0}^{\frac{2\delta}{\varepsilon}} \frac{r^{\frac{6n}{n+1}+n-1}}{(1+r^{2})^{\frac{n(n-1)}{n+1}}}\ dr \Big)^{\frac{n+1}{2n}}\notag\\
&\leq C \left\{\begin{array}{ll}
\varepsilon^{\frac{n-1}{2}}, & 3\leq n\leq 8\\
\\
\varepsilon^{4}|\ln(\varepsilon)|^{\frac{5}{9}}, & n=9\\
\\
\varepsilon^{4}, & n\geq 10
\end{array}
\right..\notag
\end{align}
For $A_{7}$, working again in conformal normal coordinates and using Proposition \ref{propest1} and \ref{proplp}, we have, $|X|=O(|x|^{2})$. Hence,
\begin{align}
\|A_{7}\|_{H^{-\frac{1}{2}}}&\leq C \Big(\int_{B_{2\delta}}|X|^{\frac{2n}{n+1}}|\varphi_{\varepsilon}|^{\frac{2n}{n+1}}\ dv \Big)^{\frac{n+1}{2n}}\notag\\
&\leq C \Big(\int_{|x|\leq 2\delta}|x|^{\frac{4n}{n+1}}|\psi_{\varepsilon}(x)|^{\frac{2n}{n+1}}\ dx \Big)^{\frac{n+1}{2n}}\notag\\
&\leq C  \varepsilon^{3} \Big(\int_{0}^{\frac{2\delta}{\varepsilon}} \frac{r^{\frac{4n}{n+1}+n-1}}{(1+r^{2})^{\frac{n(n-1)}{n+1}}}\ dx \Big)^{\frac{n+1}{2n}}\notag\\
&\leq C \left\{\begin{array}{ll}
\varepsilon^{\frac{n-1}{2}}, & 4\leq n \leq 6\\
\\
\varepsilon^{3}|\ln(\varepsilon)|^{\frac{4}{7}}, & n=7\\
\\
\varepsilon^{3}, & n\geq 8
\end{array}
\right..\notag
\end{align}
It remains now to estimate $A_{8}$. Recalling that
$$A_{8}=\sum_{i,j}(b_{ij}-\delta_{ij})\overline{\partial_{x_{i}}\cdot \nabla_{\partial_{x_{j}}}\psi_{\varepsilon}} \; ,$$
from (\ref{psiepsilon}), we will write $A_{8}=B_{1}+B_{2}$, where 
$$B_{1}:=\eta\sum_{i,j}(b_{ij}-\delta_{ij})\overline{\partial_{x_{i}}\cdot \nabla_{\partial_{x_{j}}}\Psi_{\varepsilon}}\quad \text{ and  }\quad B_{2}:=\sum_{i,j}(b_{ij}-\delta_{ij})(\partial_{x_{j}}\eta)\overline{\partial_{x_{i}}\cdot \Psi_{\varepsilon}}.$$
Notice that since $|\nabla \Psi(r)|\leq C \big(f(r)\big)^{\frac{n}{2}}$, and using Lemma \ref{lemyx}, we have
\begin{align}
\|B_{1}\|_{H^{-\frac{1}{2}}}&\leq C \Big(\int_{|x|\leq 2\delta}|x|^{\frac{6n}{n+1}}|\nabla \Psi_{\varepsilon}(x)|^{\frac{2n}{n+1}}\ dx \Big)^{\frac{n+1}{2n}}\notag\\
& \leq C \varepsilon^{3}\Big(\int_{0}^{\frac{2\delta}{\varepsilon}}\frac{r^{\frac{6n}{n+1}+n-1}}{(1+r^{2})^{\frac{n^{2}}{n+1}}}\ dr \Big)^{\frac{n+1}{2n}}\notag\\
&\leq \left\{\begin{array}{ll}
\varepsilon^{\frac{n-1}{2}}, & 1\leq n\leq 6\\
\\
\varepsilon^{3}|\ln(\varepsilon)|^{\frac{4}{7}}, & n=7\\
\\
\varepsilon^{3}, & n\geq 8
\end{array}
\right..\notag
\end{align}
We finish now by estimating $B_{2}$: 
\begin{align}
\|B_{2}\|_{H^{-\frac{1}{2}}}&\leq C \Big(\int_{|x|\leq 2\delta}|x|^{\frac{6n}{n+1}}|\Psi_{\varepsilon}(x)|^{\frac{2n}{n+1}}\ dx \Big)^{\frac{n+1}{2n}}\notag\\
&\leq C  \left\{\begin{array}{ll}
\varepsilon^{\frac{n-1}{2}}, & 3\leq n\leq 8\\
\\
\varepsilon^{4}|\ln(\varepsilon)|^{\frac{5}{9}}, & n=9\\
\\
\varepsilon^{4}, & n\geq 10
\end{array}
\right..\notag
\end{align}
All the previous estimates can be summarized as follows:
\begin{equation}\label{Reps}
\|\nabla J_{g}(\varphi_{\varepsilon})\|_{H^{-\frac{1}{2}}} = \| \overline{D}\varphi_{\varepsilon}-(V_{g}*|\varphi_{\varepsilon}|^{2})\varphi_{\varepsilon} \|_{H^{-\frac{1}{2}}} \leq \left\{\begin{array}{ll}
O(\varepsilon^{\frac{n-1}{2}}), & 4\leq n\leq 6\\
\\
O(\varepsilon^{3}|\ln(\varepsilon)|^{\frac{4}{7}}), & n=7\\
\\
O(\varepsilon^{3}), & n\geq 8
\end{array}
\right. 
\end{equation}
which is the desired result.
\end{proof}

\noindent
This last Proposition provides us with two important information. On one hand, it tells us that $(\varphi_{\varepsilon})_{\varepsilon}$ is a $(PS)$ sequence for $J_{g}$ and hence Propositions \ref{prop diff} and \ref{prop exp} can be applied. On the other hand, it provides us with a precise estimate on the decay of $\|\nabla J_{\lambda}(\varphi_{\varepsilon})\|_{H^{-\frac{1}{2}}}$. \\
So now, we proceed with the energy estimate. We will consider two cases.

\subsection{\texorpdfstring{The case $n=4,5$ or $n\geq 6$ with $M$ not locally conformally flat}{The case n=4,5 or n greater than 6 with M not locally conformally flat}}
We will assume that if $n\geq 6$, then there exists $p\in M$ such that $|W(p)|^{2}\not=0$, which means that the manifold is not locally conformally flat. Then we have the following estimates:
\begin{proposition}\label{lemma 2 exp}
For $\varphi_{\varepsilon}$ defined as above and centered at $p\in M$, we have 

\begin{equation}\label{eqj}
J_{g}(\varphi_{\varepsilon})\leq \overline{Y}(S^{n},[g_{0}])-C\left\{\begin{array}{lll}
A(p)\varepsilon^{2} +O(\varepsilon^{3}), & n=4\\
\\
A(p)\varepsilon^{3}+O(\varepsilon^{4}), & n=5\\
\\
|W(p)|^{2}\varepsilon^{4}|\ln(\varepsilon)|+O(\varepsilon^{4}), & n=6\\
\\
|W(p)|^{2}\varepsilon^{4}+o(\varepsilon^{4}), & n\geq 7
\end{array}
\right. .
\end{equation}
\end{proposition}

\begin{proof}
Recall  that in conformal normal coordinates (Proposition \ref{proplp}) we can assume that the volume form takes the form $dv_{g}=(1+O(|x|^{N}))dx$ around $p$, where $N\geq 2$. First, we start by estimating the $L^{2}$ norm of $\varphi_{\varepsilon}$. Indeed, we have

\begin{align}
\int_{M}|\varphi_{\varepsilon}|^{2}\ dv_{g}&= \int_{B_{2\delta}}|\varphi_{\varepsilon}|^{2}\ dv_{g}\notag\\
&=\int_{|x|\leq \delta}|\Psi_{\varepsilon}(x)|^{2}\ dx + \int_{\delta\leq |x|\leq 2\delta}|\eta(x)\Psi_{\varepsilon}(x)|^{2}\ dx +O\Big( \int_{|x|\leq \delta}|x|^{N}|\Psi_{\varepsilon}(x)|^{2}\ dx\Big)\notag\\
&=\varepsilon a_{n}\omega_{n-1}\int_{0}^{\frac{\delta}{\varepsilon}}\frac{r^{n-1}}{(1+r^{2})^{n-1}}\ dr +O\Big(\varepsilon \int_{\frac{\delta}{\varepsilon}}^{\frac{2\delta}{\varepsilon}}\frac{r^{n-1}}{(1+r^{2})^{n-1}}\ dr\Big)\notag\\
&\quad+O\Big(\varepsilon^{N+1}\int_{0}^{\frac{2\delta}{\varepsilon}} \frac{r^{n+N-1}}{(1+r^{2})^{n-1}}\ dr \Big)\notag\\
&=Q_{n}\varepsilon+O(\varepsilon^{n-1})+ \left\{\begin{array}{ll}
O(\varepsilon^{n-1}), & n< N+2\\
\\
O(\varepsilon^{N+1}|\ln(\varepsilon)|), & n=N+2\\
\\
O(\varepsilon^{N+1}), & n\geq N+2
\end{array}
\right.\notag\\
&=Q_{n}\varepsilon+ \left\{\begin{array}{ll}
O(\varepsilon^{n-1}), & n<N+2\\
\\
O(\varepsilon^{N+1}|\ln(\varepsilon)|), & n=N+2\\
\\
O(\varepsilon^{N+1}), & n\geq N+2
\end{array}
\right..\label{L2norm}
\end{align}
Here, $$Q_{n}:=a_{n} \omega_{n-1}\int_{0}^{\infty} \frac{r^{n-1}}{(1+r^{2})^{n-1}}\ dr.$$
Next, we estimate $\int_{M}\langle \overline{D} \varphi_{\varepsilon},\varphi_{\varepsilon}\rangle \ dv_{g}$. Using the same decomposition as in $(\ref{decomp})$, we see that
\begin{align}
\int_{M}\langle \overline{D} \varphi_{\varepsilon},\varphi_{\varepsilon}\rangle \ dv_{g}&= \int_{M}\int_{\R^{n}}V_{g_{\R^{n}}}(x,y)|\Psi_{\varepsilon}(y)|^{2} |\varphi_{\varepsilon}(x)|^{2} \ dy \ dv_{g}(x)+\int_{M}\eta^{2} \langle \Omega\cdot \overline{\Psi_{\varepsilon}},\overline{\Psi_{\varepsilon}}\rangle \ dv_{g}\notag\\
&\quad + \int_{M}\sum_{i,j}(b_{ij}-\delta_{ij}) \eta^{2} \langle \overline{\partial_{x_{i}}\cdot \nabla_{\partial_{x_{j}}}\Psi_{\varepsilon}}, \overline{\Psi_{\varepsilon}}\rangle \ dv_{g}\notag\\
&=F_{1}+F_{2}+F_{3}.\notag
\end{align}
We will estimate each term individually starting by $F_{1}$. We have
\begin{align}
F_{1}&=\int_{M}\int_{\R^{n}} V_{g_{\R^{n}}}(x,y)|\Psi_{\varepsilon}(y)|^{2} \eta (x)|\overline{\Psi_{\varepsilon}}(x)|^{2} \ dy \ dv_{g}(x)\notag\\
&=\int_{|x|\leq \delta}\int_{\R^{n}} V_{g_{\R^{n}}}(x,y)|\Psi_{\varepsilon}(y)|^{2} |\Psi_{\varepsilon}(x)|^{2} \ dy dx\notag\\
&\quad+\int_{\delta\leq |x|\leq 2\delta}\int_{\R^{n}} V_{g_{\R^{n}}}(x,y)|\Psi_{\varepsilon}(y)|^{2} |\eta(x)|^{2} |\Psi_{\varepsilon}(x)|^{2} \ dy dx \notag\\
&\quad+\int_{|x|\leq 2\delta}\int_{\R^{n}} V_{g_{\R^{n}}}(x,y)|\Psi_{\varepsilon}(y)|^{2} |x|^{N} |\Psi_{\varepsilon}(x)|^{2} \ dy dx.\notag
\end{align}
But recall that
$$\int_{\R^{n}} V_{g_{\R^{n}}}(x,y)|\Psi_{\varepsilon}(y)|^{2}\ dy=c_{n}^{\frac{1}{n-1}}|\Psi_{\varepsilon}(x)|^{\frac{2}{n-1}},$$
where $c_{n}$ is the constant defined in $(\ref{cn})$. Hence,
\begin{align}
\int_{|x|\leq \delta}\int_{\R^{n}}V_{g_{\R^{n}}}(x,y)|\Psi_{\varepsilon}(y)|^{2} |\Psi_{\varepsilon}(x)|^{2} \ dy dx&=c_{n}^{\frac{1}{n-1}}a_{n}^{\frac{n}{n-1}}\omega_{n-1}\int_{0}^{\frac{\delta}{\varepsilon}}\frac{r^{n-1}}{(1+r^{2})^{n}}\ dr\notag\\
&=c_{n}^{\frac{1}{n-1}}a_{n}^{\frac{n}{n-1}}\omega_{n-1}\int_{0}^{+\infty}\frac{r^{n-1}}{(1+r^{2})^{n}}\ dr +O(\varepsilon^{n}).\label{Gpsi}
\end{align}
On the other hand, 
\begin{align}
\int_{\delta\leq |x|\leq 2\delta}\int_{\R^{n}}V_{g_{\R^{n}}}(x,y)|\Psi_{\varepsilon}(y)|^{2} |\eta(x)|^{2} |\Psi_{\varepsilon}(x)|^{2} \ dy dx &\leq C\int_{\frac{\delta}{\varepsilon}}^{\frac{2\delta}{\varepsilon}}\frac{r^{n-1}}{(1+r^{2})^{n}}\ dr=O(\varepsilon^{n}).\label{Gpsi2}
\end{align}
And to finish, we have
\begin{align}
\int_{|x|\leq 2\delta}\int_{\R^{n}}V_{g_{\R^{n}}}(x,y)|\Psi_{\varepsilon}(y)|^{2} |x|^{N}|\Psi_{\varepsilon}(x)|^{2} \ dy dx&=O\Big(\varepsilon^{N}\int_{0}^{\frac{2\delta}{\varepsilon}}\frac{r^{n+N-1}}{(1+r^{2})^{n}}\ dr \Big)\notag\\
&=\left\{\begin{array}{lll}
O(\varepsilon^{n}), & n<N\\
\\
O(\varepsilon^{N}|\ln(\varepsilon)|), & n=N\\
\\
O(\varepsilon^{N}), & n>N
\end{array}
\right. .\notag
\end{align}
Therefore,
$$F_{1}=c_{n}^{\frac{1}{n-1}}a_{0}^{\frac{n}{n-1}}\omega_{n-1}\int_{0}^{+\infty}\frac{r^{n-1}}{(1+r^{2})^{n}}\ dr + \left\{\begin{array}{lll}
O(\varepsilon^{n}), & n<N\\
\\
O(\varepsilon^{N}|\ln(\varepsilon)|), & n=N\\
\\
O(\varepsilon^{N}), & n>N
\end{array}
\right..$$
In order to estimate $F_{2}$, we first write
$$\Omega=\sum_{\begin{array}{cc}
i,j,k\\
i\not=j\not=k\not=i
\end{array}} A_{ijk\alpha\beta\gamma} x^{\alpha}x^{\beta}x^{\gamma}e_{i}\cdot e_{j}\cdot e_{k} +O(|x|^{4}),$$
where $A_{ijk\alpha\beta\gamma}=-\frac{1}{144}\sum\limits_{\ell} R_{\ell\beta\gamma k}(R_{ji\alpha \ell}+R_{j\ell\alpha i})$. In particular, using the expression of $\Psi_{\varepsilon}$,
\begin{align}
\sum_{\begin{array}{cc}
i,j,k\\
i\not=j\not=k\not=i
\end{array}} A_{ijk\alpha\beta\gamma}&\int_{|x|\leq \delta}  x^{\alpha}x^{\beta}x^{\gamma}\langle e_{i}\cdot e_{j}\cdot e_{k}\cdot \Psi_{\varepsilon},\Psi_{\varepsilon}\rangle \ dx\notag\\
&=\varepsilon^{4}\sum_{\ell} \tilde{A}_{ijk\ell}\langle \partial_{x_{i}}\cdot \partial_{x_{j}} \cdot \partial_{x_{k}}\cdot \partial_{x_{\ell}} \cdot \Psi_{0},\Psi_{0}\rangle \notag.
\end{align}
Therefore, using Lemma \ref{lemyx}, we have
$$F_{2}\leq C\int_{|x|\leq 2\delta}|x|^{4}|\Psi_{\varepsilon}(x)|^{2}\ dx.$$
Therefore, we conclude that
$$F_{2}\leq C\left\{\begin{array}{ll}
\varepsilon^{n-1}, & 4\leq n\leq 5\\
\\
\varepsilon^{5}|\ln(\varepsilon)|, & n=6\\
\\
\varepsilon^{5}, & n\geq 7
\end{array}
\right. .
$$
Finally for $F_{3}$ we need to use the following estimate proved in \cite{YX1}, which relies on Lemma \ref{lemyx}:
$$F_{3}\leq -C|W(p)|^{2}\left\{ \begin{array}{ll}
\varepsilon^{4}|\ln(\varepsilon)|, & n=4\\
\\
\varepsilon^{4}, & n\geq 5
\end{array}
\right.   + 
\left\{ \begin{array}{ll}
O(\varepsilon^{n-1}), & 4\leq n \leq 5\\
\\
O(\varepsilon^{5}), & n\geq 6
\end{array}
\right. \; .$$

\noindent
Therefore, if $N=5$, we can conclude that
\begin{equation}\label{dep}
\int_{M}\langle \overline{D} \varphi_{\varepsilon},\varphi_{\varepsilon}\rangle \ dv_{g}\leq c_{n}^{\frac{1}{n-1}}a_{n}^{\frac{n}{n-1}}\omega_{n-1}\int_{0}^{+\infty}\frac{r^{n-1}}{(1+r^{2})^{n}}\ dr -C |W(p)|^{2}\left\{\begin{array}{ll}\varepsilon^{4}|\ln(\varepsilon)| +O(\varepsilon^{3}), & n=4\\
\\
\varepsilon^{4}+O(\varepsilon^{4}), & n= 5\\
\\
\varepsilon^{4}+o(\varepsilon^{4}), & n\geq 6
\end{array}
\right.
\end{equation}

\noindent
We focus now on estimating the term
\begin{align}
& \qquad \int_{M}\int_{M}V_{g}(x,y)|\varphi_{\varepsilon}(y)|^{2}|\varphi_{\varepsilon}(x)|^{2}\ dv(y)\ dv(x) \notag\\
& = \int_{B_{p}(2\delta)}\int_{B_{p}(2\delta)}V_{g}(x,y)|\psi_{\varepsilon}(y)|^{2}|\psi_{\varepsilon}(x)|^{2}\ dy\ dx \notag\\
& +\int_{B_{p}(2\delta)}\int_{B_{p}(2\delta)}V_{g}(x,y)|y|^{N}|\psi_{\varepsilon}(y)|^{2}|x|^{N}|\psi_{\varepsilon}(x)|^{2}\ dy\ dx\notag\\
&=\int_{|x|\leq 2\delta}\int_{|y|\leq 4\delta}V_{g_{\R^{n}}}(0,y)|\psi_{\varepsilon}(x-y)|^{2}|\psi_{\varepsilon}(x)|^{2}\ dy \ dx \notag \\
&+\int_{|x|\leq 2\delta}\int_{|y|\leq 4\delta}h(y)|\psi_{\varepsilon}(x-y)|^{2}|\psi_{\varepsilon}(x)|^{2}\ dy \ dx\notag\\
&+\int_{B_{p}(2\delta)}\int_{B_{p}(2\delta)}V_{g}(x,y)|y|^{N}|\psi_{\varepsilon}(y)|^{2}|x|^{N}|\psi_{\varepsilon}(x)|^{2}\ dy \ dx\notag\\
&=H_{1}+H_{2}+H_{3}.\notag\\
\end{align}

\noindent
We first start by estimating $H_{3}=\int_{B_{p}(2\delta)}\int_{B_{p}(2\delta)}V_{g}(x,y)|y|^{N}|\psi_{\varepsilon}(y)|^{2}|x|^{N}|\psi_{\varepsilon}(x)|^{2}\ dy \ dx$. Using the expression of $\psi_{\varepsilon}$ in terms of $\Psi$ (\ref{psiepsilon}) and Young's inequality, we see that
$$H_{3}\leq C\left\{ \begin{array}{ll}
\varepsilon^{2(n-1)}, & n<N+1\\
\\
\varepsilon^{2N}|\ln(\varepsilon)|^{2}, & n=N+1\\
\\
\varepsilon^{2N}, & n>N+1
\end{array}
\right..$$
For $H_{2}$ we distinguish three cases, depending on the expansion of the Green's function. Notice that the function $F:x\mapsto \int_{|y|\leq 2\delta}h(y)|\psi_{\varepsilon}(x-y)|^{2}\ dy$ is continuous  for $|x|\leq \delta$. We will estimate its value around $x=0$ corresponding to the point $p\in M$.
\begin{itemize}
\item The case $n\geq 7$:
\end{itemize}
In this case, there exist positive constants $k_{n,1}$ and $k_{n,2}$ (that can be computed explicitly from Proposition \ref{propgreen}) such that 
$$h(y)=k_{n,1}|y|^{2}|W(p)|^{2}-k_{n,2}R_{,ij}y^{i}y^{j}+O(|y|^{3}) \; .$$
Thus, there exist two positive constants $k_{n,3}$ and $k_{n,4}$ such that
$$ F(0)=k_{n,3}|W(p)|^{2}\int_{|y|\leq 2\delta}|y|^{2}|\psi_{\varepsilon}(y)|^{2}\ dy=k_{n,4}|W(p)|^{2}\varepsilon^{3}.$$
Taking $\delta$ even smaller if necessary, we can assume that for $|x|\leq 2\delta$, $F(x)\geq \frac{k_{n,4}}{2}|W(p)|^{2}\varepsilon^{3}$. Hence, there exists a positive constant $k_{n,5}$ for which we have
$$H_{2}\geq k_{n,5}|W(p)|^{2}\varepsilon^{4}+o(\varepsilon^{4}).$$

\begin{itemize}
\item The case $n=6$:
\end{itemize}
In this case, there exists a positive constant $k_{n,1}$ such that
$$h(y)=-k_{n,1}|W(p)|^{2}|y|^{2}\ln(|y|)+O(|y|^{2}).$$
Therefore, there exist $k_{n,2}$ and $k_{n,3}$, two positive constants, so that
\begin{align}
F(0)&=-k_{n,2}|W(p)|^{2}\int_{|y|\leq 2\delta}|y|^{2}\ln(|y|)|\psi_{\varepsilon}(y)|^{2}\ dy\notag\\
&=-k_{n,3}|W(p)|^{2}\varepsilon^{4}\ln(\varepsilon)+O(\varepsilon^{4})\notag
\end{align}
Thus, choosing $\delta$ even smaller if necessary, we have 
$$H_{2}\geq k_{n,4}|W(p)|^{2}\varepsilon^{4}|\ln(\varepsilon)|+O(\varepsilon^{4}),$$
where $k_{n,4}$ is another positive constant.

\begin{itemize}
\item The case $n=5$ or $n=4$:
\end{itemize}
In these two cases we have, for a positive constant $k_{n,1}$
$$h(y)=k_{n,1}\left\{\begin{array}{ll}
A(p)|y|+O(|y|^{2}), & n=5\\
\\
A(p)+O(|y|), & n=4
\end{array}
\right..$$
Thus, following the same process, we have that for a positive constant $k_{n,2}$
$$H_{2}\geq k_{n,2}A(p)\left\{\begin{array}{ll}
\varepsilon^{3}+O(\varepsilon^{3}), & n=5\\
\\
\varepsilon^{2}+O(\varepsilon^{2}), & n=4
\end{array}
\right..$$
To finish the proof, we need to estimate $H_{1}$. Indeed,

\begin{align}
H_{1}&=\int_{|x|\leq \delta}\int_{|y|\leq \delta}V_{g_{\R^{n}}}(x,y)|\Psi_{\varepsilon}(y)|^{2}|\Psi_{\varepsilon}(x)|^{2}\ dy \ dx\notag\\
&\quad+ 2\int_{|x|\leq \delta}\int_{\delta \leq |y|\leq 2\delta}V_{g_{\R^{n}}}(x,y)|\psi_{\varepsilon}(y)|^{2}|\Psi_{\varepsilon}(x)|^{2}\ dy \ dx \notag\\
&\quad+\int_{\delta\leq|x|\leq 2\delta}\int_{\delta \leq |y|\leq 2\delta}V_{g_{\R^{n}}}(x,y)|\psi_{\varepsilon}(y)|^{2}|\psi_{\varepsilon}(x)|^{2}\ dy \ dx\notag\\
&=L_{1}+L_{2}+L_{3}.\notag
\end{align}
Notice that using Young's inequality we have
\begin{align}
L_{3} \leq C\|\Psi_{\varepsilon}\|_{L^{\frac{2n}{n-1}}(B_{2\delta}\setminus B_{\delta})}^{4} \leq C\Big( \int_{\frac{\delta}{\varepsilon}}^{\frac{2\delta}{\varepsilon}}\frac{r^{n-1}}{(1+r^{2})^{n}}\ dr \Big)^{\frac{2(n-1)}{n}}
 \leq C\varepsilon^{2(n-1)}.\notag
\end{align}
On the other hand, we have
\begin{align}
L_{2}&\leq C\|\Psi_{\varepsilon}\|_{L^{\frac{2n}{n-1}}(B_{2\delta}\setminus B_{\delta})}^{2}\|\Psi_{\varepsilon}\|_{L^{\frac{2n}{n-1}}}^{2} \leq C\varepsilon^{n-1}\notag
\end{align}
For $L_1$ we can easily see that
\begin{align}
L_{1}&=\int_{\R^{n}}\int_{\R^{n}}V_{g_{\R^{n}}}(x,y)|\Psi_{\varepsilon}(y)|^{2}|\Psi_{\varepsilon}(x)|^{2} \ dy \ dx +O(\varepsilon^{n-1})\notag\\
&=c_{n}^{\frac{1}{n-1}}a_{n}^{\frac{n}{n-1}}\omega_{n-1}\int_{0}^{+\infty}\frac{r^{n-1}}{(1+r^{2})^{n}}\ dr +O(\varepsilon^{n-1}).\notag
\end{align}
Hence,
$$H_{1}=c_{n}^{\frac{1}{n-1}}a_{n}^{\frac{n}{n-1}}\omega_{n-1}\int_{0}^{+\infty}\frac{r^{n-1}}{(1+r^{2})^{n}}\ dr+O(\varepsilon^{n-1}).$$
Therefore, we conclude that
\begin{align}\label{vep}
& \int_{M}\int_{M} V_{g}(x,y)|\varphi_{\varepsilon}(y)|^{2}|\varphi_{\varepsilon}(x)|^{2}\ dv(y)\ dv(x) \geq \notag\\
& c_{n}^{\frac{1}{n-1}}a_{n}^{\frac{n}{n-1}}\omega_{n-1}\int_{0}^{+\infty}\frac{r^{n-1}}{(1+r^{2})^{n}}\ dr
+ C\left\{ \begin{array}{lll}
A(p)\varepsilon^{2}+O(\varepsilon^{2}), & n=4\\
\\
A(p)\varepsilon^{3}+O(\varepsilon^{3}), & n=5\\
\\
|W(p)|^{2}\varepsilon^{4}|\ln(\varepsilon)|+O(\varepsilon^{4}), & n=6\\
\\
|W(p)|^{2}\varepsilon^{4} + o(\varepsilon^{4}), & n\geq 7
\end{array}
\right.
\end{align}
Thus, we can obtain $(\ref{eqj})$ from $(\ref{dep})$ and $(\ref{vep})$.
\end{proof}

\noindent
We are now in position to state Theorem \ref{main} in this case:

\begin{theorem}
Let $(M,g)$ be a closed Riemannian spin manifold of dimension $n$, with positive Yamabe invariant $Y(M,[g])>0$. Let us assume $n=4,5$ or $n\geq 6$ with $M$ not locally conformally flat. Then $\overline{Y}(M,[g])<\overline{Y}(S^{n},[g_{0}])$. In particular, the equation $(\ref{eqv})$ has a non-trivial ground state solution.
\end{theorem}
\begin{proof}
We recall that from Proposition \ref{prop exp}, we have for small $\varepsilon>0$ and with $\delta_0$ as in (\ref{delta0}),
$$\delta_{0}\leq J_{g}(\varphi_{\varepsilon})+O(\|\nabla J_{g}(\varphi_{\varepsilon})\|^{2}).$$
Therefore, using Proposition \ref{lemma1 exp} and \ref{lemma 2 exp}, we have:
$$\delta_{0}\leq \overline{Y}(S^{n},[g_{0}])-C\left\{\begin{array}{lll}
A(p)\varepsilon^{2} +O(\varepsilon^{3}), & n=4\\
\\
A(p)\varepsilon^{3}+O(\varepsilon^{4}), & n=5\\
\\
|W(p)|^{2}\varepsilon^{4}|\ln(\varepsilon)|+O(\varepsilon^{4}), & n=6\\
\\
|W(p)|^{2}\varepsilon^{4}+o(\varepsilon^{4}), & n\geq 7
\end{array}
\right. .
$$
In particular, we have
$$\overline{Y}(M,[g])= \delta_{0}<\overline{Y}(S^{n},[g_{0}]) \; .$$
Therefore, using Corollary \ref{cor1} we deduce that $(\ref{eqv})$ has a non-trivial ground-state solution.
\end{proof}

\subsection{\texorpdfstring{The case $n\geq 6$ with $M$ locally conformally flat}{The case n greater than 6 with M locally conformally flat}}

The locally conformally flat case is slightly easier than the previous one. This is in fact surprising when compared to the classical Yamabe problem, where one needs to choose a different test function with more control on the tails of its energy. In fact, we are closer to the framework of \cite{HV}, rather than that of \cite{S}. In our setting, we will be using the same test function defined in $(\ref{psiepsilon})$. Though, since the manifold is locally conformally flat, we can assume that it is flat in the neighborhood of a point $p\in M$. In this flat coordinate patch, we see that the quantities $\Omega$ and $X$ vanish identically. Moreover, $b_{ij}=\delta_{ij}$. Hence, after identifying $p$ with zero in a coordinate patch $B_{2\delta}$, we have that
$$D_{g}\varphi_{\varepsilon}=\Big(\int_{\R^{n}}V_{g_{\R^{n}}}(x,y)|\Psi_{\varepsilon}(y)|^{2}\ dy  \Big)\psi_{\varepsilon} +\nabla \eta \cdot \Psi_{\varepsilon}(x).$$
Moreover, from Proposition \ref{propgreen}, we have that
$$G_{g}(x,y)=b_{n}\Big(\frac{1}{|x-y|^{n-2}}+A(x) + O(|x-y|)\Big),$$
where $A$ is the mass, which we know from the positive mass theorem \cite{W,PT}, that if $(M,[g])$ is not conformal to the standard sphere, then $A>0$.\\
Looking back at the proof of Proposition \ref{lemma1 exp}, we can see that for $n\geq 6$,

$$\|\nabla J_{g}(\varphi_{\varepsilon})\|_{H^{-\frac{1}{2}}}\leq C \varepsilon^{\frac{n-1}{2}}.$$
Moreover, because of the vanishing of $\Omega$ and $X$, there exists a positive $C$ such that
$$J_{g}(\varphi_{\varepsilon})\leq \overline{Y}(S^{n},[g_{0}])-CA(p)\varepsilon^{n-2}+O(\varepsilon^{n-1}).$$
Again, using Proposition \ref{prop exp}, we have for $\varepsilon>0$ small enough,
$$\overline{Y}(M,[g])<\overline{Y}(S^{n},[g_{0}]).$$
Thus we can conclude using Corollary \ref{cor1}.

\end{document}